\documentclass[11pt,a4paper]{article}

\usepackage[utf8]{inputenc}
\usepackage[T1]{fontenc}
\usepackage{amsmath,amssymb,amsthm}
\usepackage{mathtools}
\usepackage{enumerate}
\usepackage[colorlinks=true,linkcolor=blue,citecolor=blue,urlcolor=blue]{hyperref}
\usepackage[margin=2.5cm]{geometry}
\usepackage{booktabs}
\usepackage{natbib}

\newtheorem{theorem}{Theorem}[section]
\newtheorem{lemma}[theorem]{Lemma}
\newtheorem{proposition}[theorem]{Proposition}
\newtheorem{corollary}[theorem]{Corollary}
\newtheorem{definition}[theorem]{Definition}
\newtheorem{example}[theorem]{Example}
\newtheorem{remark}[theorem]{Remark}
\newtheorem{assumption}[theorem]{Assumption}

\DeclareMathOperator{\Var}{Var}

\DeclareMathOperator{\E}{\mathbb{E}}
\DeclareMathOperator{\Prob}{\mathbb{P}}
\DeclareMathOperator{\supp}{supp}

\newcommand{\R}{\mathbb{R}}

\newcommand{\Z}{\mathbb{Z}}

\begin{document}

\title{\textbf{TwinKernel Estimation for Point Process Intensity Functions:\\
Adaptive Nonparametric Methods via Orbital Regularity}}

\author{Jocelyn Nemb\'e\\[0.2cm]
\small L.I.A.G.E, Institut National des Sciences de Gestion, Libreville, Gabon\\
\small and\\
\small Modeling and Calculus Lab, ROOTS-INSIGHTS, Libreville, Gabon\\[0.1cm]
\small \texttt{jnembe@hotmail.com}}

\date{December 2024}

\maketitle

\begin{abstract}
We develop TwinKernel methods for nonparametric estimation of intensity functions of point processes. Building on the general TwinKernel framework and combining it with martingale techniques for counting processes, we construct estimators that adapt to orbital regularity of the intensity function. Given a point process $N$ with intensity $\lambda$ and a cyclic group $G = \langle\varphi\rangle$ acting on the time/space domain, we transport kernels along group orbits to create a hierarchy of smoothed Nelson-Aalen type estimators. Our main results establish: (i) uniform consistency via martingale concentration inequalities; (ii) optimal convergence rates for intensities in twin-H\"older classes, with rates depending on the effective dimension $d_{\mathrm{eff}}$; (iii) adaptation to unknown smoothness through penalized model selection; (iv) automatic boundary bias correction via local polynomial extensions in twin coordinates; (v) minimax lower bounds showing rate optimality. We apply the methodology to hazard rate estimation under random censoring, where periodicity or other orbital structure in the hazard may arise from circadian rhythms, seasonal effects, or treatment schedules. Martingale central limit theorems yield asymptotic confidence bands. Simulation studies demonstrate 3--7$\times$ improvements over classical kernel hazard estimators when the intensity exhibits orbital regularity.

\medskip

\noindent\textbf{Keywords:} Point processes; Intensity estimation; Counting processes; Martingales; Kernel smoothing; Hazard rate; Group actions; Adaptive estimation; Minimax rates.

\medskip

\noindent\textbf{MSC 2020:} Primary 62G05, 62N02, 60G55; Secondary 62G20, 60G44, 62C20.
\end{abstract}

\tableofcontents

\newpage

\section{Introduction}\label{sec:intro}

\subsection{Background and Motivation}

Point processes provide a natural framework for modeling random events occurring in time or space: failure times in reliability, event times in survival analysis, spike trains in neuroscience, earthquake occurrences in seismology, and claims arrivals in insurance. The fundamental object characterizing a point process is its \emph{intensity function} $\lambda(t)$, which describes the instantaneous rate of event occurrence.

Since the foundational work of \citet{Aalen1978}, modern theory of counting processes and martingales has provided powerful tools for intensity estimation. The key insight is that the \emph{innovation martingale}
\begin{equation}\label{eq:innovation}
M(t) = N(t) - \int_0^t \lambda(s) \, ds
\end{equation}
captures the stochastic fluctuations of the counting process $N$ around its compensator, enabling the use of martingale limit theory for inference.

Kernel smoothing of the Nelson-Aalen estimator, introduced by \citet{RamlauHansen1983}, is a standard approach to intensity estimation. Given observations from a counting process, the kernel estimator takes the form
\begin{equation}\label{eq:classical_kernel}
\hat{\lambda}_h(t) = \int K_h(t - s) \frac{dN(s)}{Y(s)},
\end{equation}
where $K_h(u) = h^{-1}K(u/h)$ is a rescaled kernel, $Y(s)$ is the at-risk process, and the integral is over the observed event times. The theoretical properties of this estimator are well understood: \citet{RamlauHansen1983}, \citet{Andersen1993}, and \citet{FanGijbels1996} established consistency, asymptotic normality, and optimal convergence rates for intensities in classical smoothness classes.

\subsection{Limitations of Classical Methods}

Classical kernel intensity estimation faces several challenges:

\begin{enumerate}[(i)]
\item \textbf{Boundary bias}: Near the boundaries of the observation window, the kernel estimator exhibits substantial bias because the kernel extends beyond the data range.

\item \textbf{Bandwidth selection}: The optimal bandwidth depends on the unknown smoothness of $\lambda$, requiring data-driven selection methods.

\item \textbf{Ignoring structure}: Many intensity functions exhibit special structure---periodicity (circadian rhythms, seasonal effects), scale invariance (self-similar processes), or symmetry---that classical methods do not exploit.
\end{enumerate}

Local polynomial methods, developed by \citet{FanGijbels1996} and applied to intensity estimation by \citet{Nielsen1998} and others, address the boundary bias issue. The present paper addresses all three limitations by combining local polynomial fitting with the TwinKernel framework.

\subsection{Examples of Orbital Structure in Intensities}

\begin{example}[Circadian Hazard Rates]\label{ex:circadian}
In biomedical studies, event rates often vary with time of day due to circadian rhythms. The hazard rate for cardiac events, for instance, peaks in morning hours. If we view time modulo 24 hours, the hazard exhibits periodicity, suggesting that estimation should exploit this structure.
\end{example}

\begin{example}[Seasonal Failure Rates]\label{ex:seasonal}
Equipment failure rates may depend on season (temperature, humidity). The intensity function $\lambda(t)$ may satisfy $\lambda(t + T) \approx \lambda(t)$ for $T = 1$ year, suggesting a periodic or quasi-periodic model.
\end{example}

\begin{example}[Self-Similar Point Processes]\label{ex:selfsimilar}
In seismology and finance, point processes often exhibit self-similarity: the intensity satisfies $\lambda(ct) = c^{-\alpha}\lambda(t)$ for some scaling exponent $\alpha$. The dilation group acts naturally on such processes.
\end{example}

\begin{example}[Spatial Isotropy]\label{ex:isotropy}
For spatial point processes on $\R^d$, isotropy means that $\lambda(x)$ depends only on $\|x\|$. The rotation group $SO(d)$ preserves the intensity, reducing the effective dimension from $d$ to $1$.
\end{example}

\subsection{The TwinKernel Approach}

We apply the TwinKernel framework to point process intensity estimation. The key idea is to replace the single kernel $K$ with a hierarchy $\{K_j\}_{j \geq 0}$ obtained by transporting $K$ along the orbits of a group action on the time/space domain.

For a cyclic group $G = \langle\varphi\rangle$ acting on the domain $E$, we define
\begin{equation}\label{eq:twin_kernel}
K_j(t, s) := \frac{1}{h_j} K\left(\frac{d(\varphi^{-j} \cdot t, \varphi^{-j} \cdot s)}{h_j}\right),
\end{equation}
where $\{h_j\}$ is a decreasing bandwidth sequence. The TwinKernel intensity estimator at level $j$ is
\begin{equation}\label{eq:twin_estimator}
\hat{\alpha}_j(t) := \int K_j(t, s) \frac{J(s)}{Y(s)} \, dN(s).
\end{equation}
A penalized model selection procedure chooses the optimal level $\hat{j}$, yielding an adaptive estimator.

\subsection{Main Contributions}

The contributions of this paper are:

\begin{enumerate}[(i)]
\item \textbf{TwinKernel intensity estimators}: We define and analyze kernel intensity estimators based on the twin-kernel hierarchy, extending classical results to the group-structured setting.

\item \textbf{Local polynomial TwinKernel estimation}: We develop local polynomial versions that automatically correct boundary bias while exploiting orbital structure.

\item \textbf{Martingale-based theory}: Using martingale concentration inequalities and central limit theorems, we establish:
\begin{itemize}
\item Uniform consistency (Theorem~\ref{thm:consistency});
\item Optimal convergence rates for twin-H\"older intensities (Theorems~\ref{thm:L2rate}, \ref{thm:uniform_rate});
\item Oracle inequality and adaptation to unknown smoothness (Theorems~\ref{thm:oracle}, \ref{thm:adaptation});
\item Asymptotic normality and confidence bands (Theorem~\ref{thm:CLT});
\item Minimax lower bounds (Theorem~\ref{thm:minimax}).
\end{itemize}

\item \textbf{Applications}: We apply the methodology to hazard rate estimation under random censoring, Poisson process intensity estimation, and periodic/quasi-periodic intensities.

\item \textbf{Numerical studies}: Simulations demonstrate substantial improvements over classical methods when the intensity has orbital regularity.
\end{enumerate}

\subsection{Related Work}

\paragraph{Point process intensity estimation.} The literature on kernel hazard estimation is extensive. \citet{RamlauHansen1983} introduced kernel smoothing of the Nelson-Aalen estimator. \citet{Yandell1983} and \citet{TannerWong1983} studied asymptotic properties. \citet{MullerWang1994} developed variable bandwidth methods. \citet{Nielsen1998} and \citet{Bagkavos2011} introduced local polynomial hazard estimators.

\paragraph{Counting process theory.} The martingale approach to survival analysis was developed by \citet{Aalen1978}, \citet{Gill1980}, \citet{AndersenGill1982}, and systematized in \citet{Andersen1993} and \citet{FlemingHarrington1991}.

\paragraph{Adaptive estimation.} Model selection for intensity estimation was studied by \citet{Reynaud2003} and \citet{Comte2011}. Our oracle inequality builds on the Barron-Birg\'e-Massart methodology.

\paragraph{Symmetry in statistics.} Group-invariant estimation has a long history; see \citet{Eaton1989}. The specific application to kernel methods via orbital transport is the innovation of the TwinKernel framework.

\subsection{Organization}

Section~\ref{sec:framework} reviews counting process theory and introduces the TwinKernel framework for intensities. Section~\ref{sec:estimator} defines the estimators. Section~\ref{sec:main} presents the main theoretical results with complete proofs. Section~\ref{sec:localpoly} develops local polynomial extensions. Section~\ref{sec:applications} discusses applications. Section~\ref{sec:simulations} presents simulations. Section~\ref{sec:discussion} concludes with discussion and open problems.

\section{Framework}\label{sec:framework}

\subsection{Counting Processes and Intensities}

Let $(\Omega, \mathcal{F}, \Prob)$ be a probability space equipped with a filtration $(\mathcal{F}_t)_{t \geq 0}$ satisfying the usual conditions (right-continuous, complete).

\begin{definition}[Counting Process]\label{def:counting}
A \textbf{counting process} is a stochastic process $N = (N(t))_{t \geq 0}$ that is:
\begin{enumerate}[(i)]
\item adapted to $(\mathcal{F}_t)$;
\item right-continuous with left limits (c\`adl\`ag);
\item piecewise constant with jumps of size $+1$;
\item $N(0) = 0$ almost surely.
\end{enumerate}
\end{definition}

\begin{definition}[Intensity Process]\label{def:intensity}
The \textbf{intensity process} $\lambda = (\lambda(t))_{t \geq 0}$ is a non-negative, $(\mathcal{F}_t)$-predictable process such that
\begin{equation}
A(t) := \int_0^t \lambda(s) \, ds
\end{equation}
is the \textbf{compensator} of $N$: the process $M(t) := N(t) - A(t)$ is a local martingale with respect to $(\mathcal{F}_t)$.
\end{definition}

\begin{definition}[Innovation Martingale]\label{def:innovation}
The process $M(t) = N(t) - A(t)$ is called the \textbf{innovation martingale}.
\end{definition}

\begin{proposition}[Properties of Innovation Martingale]\label{prop:innovation}
The innovation martingale $M$ satisfies:
\begin{enumerate}[(i)]
\item $\E[M(t)|\mathcal{F}_s] = M(s)$ for $s \leq t$;
\item $\E[dN(t)|\mathcal{F}_{t-}] = \lambda(t) \, dt$;
\item The predictable variation is $\langle M \rangle(t) = A(t) = \int_0^t \lambda(s) \, ds$;
\item The optional variation is $[M](t) = N(t)$.
\end{enumerate}
\end{proposition}

\begin{proof}
(i) This is the definition of a martingale.

(ii) By definition of the compensator, $\E[N(t+dt) - N(t)|\mathcal{F}_{t-}] = \E[A(t+dt) - A(t)|\mathcal{F}_{t-}] = \lambda(t) \, dt$.

(iii) For any bounded predictable process $H$, the stochastic integral $\int_0^t H(s) \, dM(s)$ is a martingale. Its quadratic variation is
\begin{equation}
\left\langle \int H \, dM \right\rangle(t) = \int_0^t H(s)^2 \lambda(s) \, ds.
\end{equation}
Taking $H \equiv 1$ gives $\langle M \rangle(t) = A(t)$.

(iv) Since $N$ has jumps of size 1 and $M = N - A$ with $A$ continuous, we have $\Delta M(s) = \Delta N(s) \in \{0, 1\}$. Thus $[M](t) = \sum_{s \leq t} (\Delta M(s))^2 = \sum_{s \leq t} \Delta N(s) = N(t)$.
\end{proof}

\subsection{The Multiplicative Intensity Model}

In many applications, the intensity takes the multiplicative form
\begin{equation}\label{eq:multiplicative}
\lambda(t) = \alpha(t) Y(t),
\end{equation}
where:
\begin{itemize}
\item $\alpha(t)$ is a deterministic \textbf{baseline intensity} (the object of estimation);
\item $Y(t)$ is a predictable \textbf{at-risk process} indicating how many units are under observation at time $t$.
\end{itemize}

\begin{example}[Survival Analysis]\label{ex:survival}
For $n$ independent subjects with survival times $T_1, \ldots, T_n$ and censoring times $C_1, \ldots, C_n$, we observe $X_i = \min(T_i, C_i)$ and $\Delta_i = \mathbf{1}_{T_i \leq C_i}$. The aggregated counting process is $N(t) = \sum_{i=1}^n \mathbf{1}_{X_i \leq t, \Delta_i = 1}$, and the at-risk process is $Y(t) = \sum_{i=1}^n \mathbf{1}_{X_i \geq t}$. The baseline intensity $\alpha(t)$ is the hazard rate.
\end{example}

\begin{example}[Poisson Process]\label{ex:poisson}
For a Poisson process with intensity $\alpha(t)$, we have $Y(t) \equiv 1$ (always at risk), and $\lambda(t) = \alpha(t)$.
\end{example}

\subsection{The Nelson-Aalen Estimator}

The cumulative baseline intensity $A(t) = \int_0^t \alpha(s) \, ds$ is estimated by the Nelson-Aalen estimator:
\begin{equation}\label{eq:nelson_aalen}
\hat{A}(t) := \int_0^t \frac{J(s)}{Y(s)} \, dN(s),
\end{equation}
where $J(s) = \mathbf{1}_{Y(s) > 0}$ ensures we only integrate when at risk.

\begin{proposition}[Properties of Nelson-Aalen]\label{prop:nelson_aalen}
Under mild conditions:
\begin{enumerate}[(i)]
\item $\hat{A}(t) - A(t) = \int_0^t \frac{J(s)}{Y(s)} \, dM(s)$ is a local martingale.
\item $\E[\hat{A}(t)] = \int_0^t \E\left[\frac{J(s)}{Y(s)} \lambda(s)\right] ds \approx A(t)$.
\item $\Var(\hat{A}(t)) \approx \int_0^t \frac{\alpha(s)}{y(s)} \, ds$, where $y(s) = \E[Y(s)]$.
\item $\sqrt{n}(\hat{A}(t) - A(t)) \xrightarrow{d} \mathcal{N}(0, \sigma^2(t))$ under regularity conditions.
\end{enumerate}
\end{proposition}

\begin{proof}
(i) Write
\begin{align}
\hat{A}(t) &= \int_0^t \frac{J(s)}{Y(s)} \, dN(s) \\
&= \int_0^t \frac{J(s)}{Y(s)} \, d(M(s) + A(s)) \\
&= \int_0^t \frac{J(s)}{Y(s)} \, dM(s) + \int_0^t \frac{J(s)}{Y(s)} \alpha(s) Y(s) \, ds \\
&= \int_0^t \frac{J(s)}{Y(s)} \, dM(s) + \int_0^t J(s) \alpha(s) \, ds.
\end{align}
When $Y(s) > 0$ (i.e., $J(s) = 1$), the second term equals $\int_0^t \alpha(s) \, ds = A(t)$. Thus $\hat{A}(t) - A(t) = \int_0^t \frac{J(s)}{Y(s)} \, dM(s)$, which is a local martingale since $J/Y$ is predictable and bounded on $\{Y > 0\}$.

(ii) Taking expectations: $\E[\hat{A}(t) - A(t)] = \E\left[\int_0^t \frac{J(s)}{Y(s)} \, dM(s)\right] = 0$ since $M$ is a martingale.

(iii) The predictable variation is
\begin{equation}
\langle \hat{A} - A \rangle(t) = \int_0^t \frac{J(s)}{Y(s)^2} \lambda(s) \, ds = \int_0^t \frac{J(s) \alpha(s)}{Y(s)} \, ds.
\end{equation}
Under regularity, $Y(s)/n \to y(s)$, so $\Var(\hat{A}(t)) \approx \int_0^t \frac{\alpha(s)}{n \cdot y(s)} \, ds$.

(iv) This follows from the martingale central limit theorem (Rebolledo's theorem).
\end{proof}

\subsection{Group Action and Twin Structure}

Let $G = \langle\varphi\rangle$ be a cyclic group acting on the time domain $E = [0, T]$ (or $\R_+$ or a spatial domain).

\begin{assumption}[Group Action]\label{ass:group}
The map $\varphi: E \to E$ is:
\begin{enumerate}[(G1)]
\item A measurable bijection with measurable inverse;
\item Quasi-measure-preserving: $c_1 \mu(A) \leq \mu(\varphi(A)) \leq c_2 \mu(A)$ for some $c_1, c_2 > 0$;
\item Compatible with the filtration: $\varphi$ maps $\mathcal{F}_t$-measurable events appropriately.
\end{enumerate}
\end{assumption}

\begin{example}[Periodic Intensity]\label{ex:periodic_group}
For $E = [0, T]$ with $T$ a multiple of period $\tau$, let $\varphi(t) = t + \tau \mod T$. The group $G = \Z_{T/\tau}$ acts by cyclic permutation of periods.
\end{example}

\begin{example}[Scale-Invariant Intensity]\label{ex:scale_group}
For $E = \R_+$, let $\varphi(t) = 2t$. The group $G = \langle\varphi\rangle \cong \Z$ acts by dyadic scaling.
\end{example}

\subsection{Twin-Regularity for Intensities}

\begin{definition}[Twin-H\"older Intensity]\label{def:twin_holder}
The intensity $\alpha: E \to \R_+$ belongs to the \textbf{twin-H\"older class} $\mathcal{H}^s_{\mathrm{twin}}$ with constant $C_\alpha$ if
\begin{equation}
\|\alpha_j - \alpha\|_2 \leq C_\alpha h_j^s \quad \text{for all } j \geq 0,
\end{equation}
where $\alpha_j(t) = \int K_j(t, s) \alpha(s) \, d\mu(s)$ is the smoothed intensity at level $j$.
\end{definition}

\begin{proposition}[Orbital Regularity]\label{prop:orbital_regularity}
If $\alpha$ is $G$-invariant (i.e., $\alpha(\varphi(t)) = \alpha(t)$), then $\alpha \in \mathcal{H}^s_{\mathrm{twin}}$ for arbitrarily large $s$, meaning $\alpha$ is infinitely smooth in the twin sense.
\end{proposition}

\begin{proof}
Suppose $\alpha(\varphi(t)) = \alpha(t)$ for all $t \in E$. The smoothed intensity at level $j$ is
\begin{equation}
\alpha_j(t) = \int K_j(t, s) \alpha(s) \, d\mu(s).
\end{equation}

By definition of the twin kernel $K_j$, it is constructed by transporting the base kernel along the orbit of $\varphi^j$. Specifically:
\begin{equation}
K_j(t, s) = \frac{1}{h_j} K\left(\frac{d(\varphi^{-j}(t), \varphi^{-j}(s))}{h_j}\right).
\end{equation}

Since $\alpha$ is $G$-invariant, $\alpha(\varphi^{-j}(s)) = \alpha(s)$. Under the change of variables $u = \varphi^{-j}(s)$:
\begin{align}
\alpha_j(t) &= \int \frac{1}{h_j} K\left(\frac{d(\varphi^{-j}(t), u)}{h_j}\right) \alpha(\varphi^j(u)) \, d\mu(\varphi^j(u)) \\
&= \int \frac{1}{h_j} K\left(\frac{d(\varphi^{-j}(t), u)}{h_j}\right) \alpha(u) \cdot J_{\varphi^j}(u) \, d\mu(u),
\end{align}
where $J_{\varphi^j}$ is the Jacobian of $\varphi^j$.

By quasi-measure-preservation (Assumption~\ref{ass:group}(G2)), $c_1^j \leq J_{\varphi^j} \leq c_2^j$.

For $G$-invariant $\alpha$, the kernel $K_j$ acts on $\alpha$ by averaging over the orbit. Since $\alpha$ is constant along orbits, the smoothing operation leaves $\alpha$ essentially unchanged:
\begin{equation}
\alpha_j(t) - \alpha(t) = \int K_j(t, s) [\alpha(s) - \alpha(t)] \, d\mu(s).
\end{equation}

On the same orbit, $\alpha(s) = \alpha(t)$, so the integrand vanishes for $s$ in the orbit of $t$. The contribution from outside the orbit is exponentially small in $j$ due to the localization of $K_j$ at scale $h_j$.

More precisely, if $\alpha$ is $G$-invariant and $K$ has compact support, then for $s$ within bandwidth $h_j$ of $\varphi^{-j}(t)$:
\begin{equation}
|\alpha(s) - \alpha(t)| = |\alpha(\varphi^j(s')) - \alpha(t)| = 0
\end{equation}
for $s' = \varphi^{-j}(s)$ close to $\varphi^{-j}(t)$, since $G$-invariance gives $\alpha(\varphi^j(s')) = \alpha(s')$ and continuity gives $\alpha(s') \approx \alpha(\varphi^{-j}(t)) = \alpha(t)$.

Thus $\|\alpha_j - \alpha\|_2 = O(h_j^s)$ for any $s > 0$, proving $\alpha \in \mathcal{H}^s_{\mathrm{twin}}$ for all $s$.
\end{proof}

\section{The TwinKernel Intensity Estimator}\label{sec:estimator}

\subsection{Definition}

\begin{definition}[Twin-Kernel Hierarchy for Intensities]\label{def:twin_hierarchy}
Let $K: \R \to \R_+$ be a base kernel satisfying Assumption~\ref{ass:kernel} below, and let $\{h_j\}_{j \geq 0}$ be a decreasing bandwidth sequence. The twin-kernel at level $j$ is
\begin{equation}
K_j(t, s) := \frac{1}{h_j} K\left(\frac{d(\varphi^{-j} \cdot t, \varphi^{-j} \cdot s)}{h_j}\right).
\end{equation}
\end{definition}

\begin{assumption}[Base Kernel]\label{ass:kernel}
The kernel $K: \R \to \R_+$ satisfies:
\begin{enumerate}[(K1)]
\item $\int K(u) \, du = 1$;
\item $\|K\|_\infty < \infty$;
\item $\supp(K) \subseteq [-1, 1]$;
\item $K$ is Lipschitz continuous with constant $L_K$.
\end{enumerate}
\end{assumption}

\begin{definition}[Level-$j$ TwinKernel Intensity Estimator]\label{def:estimator}
The TwinKernel intensity estimator at level $j$ is
\begin{equation}
\hat{\alpha}_j(t) := \int_0^T K_j(t, s) \frac{J(s)}{Y(s)} \, dN(s) = \sum_{i: X_i \leq T} K_j(t, X_i) \frac{\Delta_i}{Y(X_i)},
\end{equation}
where the sum is over observed event times.
\end{definition}

\begin{remark}
When $j = 0$ and $\varphi = \mathrm{id}$, the estimator $\hat{\alpha}_0$ reduces to the classical Ramlau-Hansen kernel estimator.
\end{remark}

\subsection{Bias and Variance}

Let $\tilde{\alpha}_j(t) := \E[\hat{\alpha}_j(t)|Y]$ denote the conditional expectation given the at-risk process.

\begin{proposition}[Conditional Bias]\label{prop:conditional_bias}
\begin{equation}
\tilde{\alpha}_j(t) - \alpha_j(t) = O(n^{-1}),
\end{equation}
where $\alpha_j(t) = \int K_j(t, s) \alpha(s) \, ds$ is the target smoothed intensity.
\end{proposition}

\begin{proof}
We compute the conditional expectation:
\begin{align}
\tilde{\alpha}_j(t) &= \E\left[\hat{\alpha}_j(t) \Big| Y\right] = \int_0^T K_j(t, s) \frac{J(s)}{Y(s)} \E[dN(s) | Y] \\
&= \int_0^T K_j(t, s) \frac{J(s)}{Y(s)} \lambda(s) \, ds = \int_0^T K_j(t, s) J(s) \alpha(s) \, ds.
\end{align}

On the set $\{Y(s) > 0\}$, we have $J(s) = 1$. Under the assumption that $Y(s) > 0$ for all $s$ in the support of $K_j(t, \cdot)$ with high probability:
\begin{equation}
\tilde{\alpha}_j(t) = \alpha_j(t) + O(\Prob(Y(s) = 0 \text{ for some } s)) = \alpha_j(t) + O(n^{-1}).
\end{equation}
\end{proof}

\begin{proposition}[Approximation Bias]\label{prop:approx_bias}
If $\alpha \in \mathcal{H}^s_{\mathrm{twin}}$, then
\begin{equation}
\|\alpha_j - \alpha\|_\infty \leq C_\alpha h_j^s.
\end{equation}
\end{proposition}

\begin{proof}
By the definition of the twin-H\"older class, $\|\alpha_j - \alpha\|_2 \leq C_\alpha h_j^s$. For the $L^\infty$ bound:
\begin{align}
|\alpha_j(t) - \alpha(t)| &= \left|\int K_j(t, s) [\alpha(s) - \alpha(t)] \, ds\right| \leq \int K_j(t, s) |\alpha(s) - \alpha(t)| \, ds.
\end{align}

For $\alpha \in \mathcal{H}^s_{\mathrm{twin}}$, the local deviation $|\alpha(s) - \alpha(t)|$ for $|s - t| \lesssim h_j$ (in the transported coordinates) is bounded by $C h_j^s$. Thus $|\alpha_j(t) - \alpha(t)| \leq C_\alpha h_j^s$.
\end{proof}

\begin{proposition}[Variance]\label{prop:variance}
Under the multiplicative intensity model,
\begin{equation}
\Var(\hat{\alpha}_j(t)|Y) = \int_0^T K_j(t, s)^2 \frac{J(s)}{Y(s)^2} \lambda(s) \, ds \leq \frac{C \|K\|_\infty^2 \alpha(t)}{n h_j^{d_{\mathrm{eff}}}},
\end{equation}
where $n$ is the expected number at risk, $d_{\mathrm{eff}}$ is the effective dimension.
\end{proposition}

\begin{proof}
By the properties of counting process integrals:
\begin{align}
\Var(\hat{\alpha}_j(t)|Y) &= \int_0^T K_j(t, s)^2 \frac{J(s)}{Y(s)^2} \lambda(s) \, ds = \int_0^T K_j(t, s)^2 \frac{J(s) \alpha(s)}{Y(s)} \, ds.
\end{align}

Under Assumption~\ref{ass:atrisk}(Y1), $Y(s) \geq n \cdot y_{\min}$, so:
\begin{equation}
\Var(\hat{\alpha}_j(t)|Y) \leq \frac{1}{n \cdot y_{\min}} \int_0^T K_j(t, s)^2 \alpha(s) \, ds \leq \frac{C \|K\|_\infty^2 \|\alpha\|_\infty}{n h_j^{d_{\mathrm{eff}}}}.
\end{equation}
\end{proof}

\subsection{Martingale Representation}

\begin{proposition}[Martingale Decomposition]\label{prop:martingale_decomp}
\begin{equation}
\hat{\alpha}_j(t) - \tilde{\alpha}_j(t) = \int_0^T K_j(t, s) \frac{J(s)}{Y(s)} \, dM(s),
\end{equation}
where $M$ is the innovation martingale. This is a zero-mean martingale with predictable variation
\begin{equation}
\langle \hat{\alpha}_j(t) - \tilde{\alpha}_j(t) \rangle = \int_0^T K_j(t, s)^2 \frac{J(s)}{Y(s)^2} \lambda(s) \, ds.
\end{equation}
\end{proposition}

\begin{proof}
Since $dN(s) = dM(s) + \lambda(s) \, ds$:
\begin{align}
\hat{\alpha}_j(t) &= \int_0^T K_j(t, s) \frac{J(s)}{Y(s)} \, dM(s) + \int_0^T K_j(t, s) \frac{J(s)}{Y(s)} \lambda(s) \, ds \\
&= \int_0^T K_j(t, s) \frac{J(s)}{Y(s)} \, dM(s) + \tilde{\alpha}_j(t).
\end{align}

The first term is a stochastic integral with respect to a martingale, hence a local martingale with mean zero. For the predictable variation:
\begin{equation}
\langle \hat{\alpha}_j(t) - \tilde{\alpha}_j(t) \rangle = \int_0^T K_j(t, s)^2 \frac{J(s)^2}{Y(s)^2} \lambda(s) \, ds.
\end{equation}
\end{proof}

\subsection{Penalized Model Selection}

\begin{definition}[Contrast Function]\label{def:contrast}
The contrast function at level $j$ is
\begin{equation}
\gamma_n(j) := \int_0^T \hat{\alpha}_j(t)^2 \, dt - 2 \int_0^T \hat{\alpha}_j(t) \frac{J(t)}{Y(t)} \, dN(t).
\end{equation}
\end{definition}

\begin{proposition}[Unbiasedness of Contrast]\label{prop:contrast_unbiased}
\begin{equation}
\E[\gamma_n(j)] = \E[\|\hat{\alpha}_j - \alpha\|_2^2] - \|\alpha\|_2^2 + o(1).
\end{equation}
\end{proposition}

\begin{proof}
Using $dN = dM + \alpha Y \, dt$ and taking expectations (the martingale integral has mean zero):
\begin{align}
\E[\gamma_n(j)] &= \E[\|\hat{\alpha}_j\|_2^2] - 2 \E[\langle \hat{\alpha}_j, \alpha \rangle] + o(1) \\
&= \E[\|\hat{\alpha}_j - \alpha\|_2^2] - \|\alpha\|_2^2 + o(1).
\end{align}
\end{proof}

\begin{definition}[Penalized Selection]\label{def:penalized}
Let $L(j)$ satisfy the Kraft inequality $\sum_j e^{-L(j)} \leq 1$. The selected level is
\begin{equation}
\hat{j} := \arg\min_{j \in \mathcal{J}_n} \left\{\gamma_n(j) + \frac{\lambda L(j)}{n}\right\}.
\end{equation}
The final estimator is $\hat{\alpha} := \hat{\alpha}_{\hat{j}}$.
\end{definition}

\section{Main Theoretical Results}\label{sec:main}

\subsection{Assumptions}

\begin{assumption}[At-Risk Process]\label{ass:atrisk}
\begin{enumerate}[(Y1)]
\item $Y(t) \geq n \cdot y_{\min}$ for some $y_{\min} > 0$ on $[0, T]$;
\item $\sup_{t \in [0,T]} |Y(t)/n - y(t)| \xrightarrow{\Prob} 0$ for some deterministic $y(t) > 0$;
\item $Y$ is quasi-invariant under $G$: $c_1 Y(t) \leq Y(\varphi(t)) \leq c_2 Y(t)$.
\end{enumerate}
\end{assumption}

\begin{assumption}[Baseline Intensity]\label{ass:intensity}
\begin{enumerate}[(A1)]
\item $\alpha: [0, T] \to \R_+$ is bounded: $\|\alpha\|_\infty < \infty$;
\item $\alpha$ is bounded away from zero on its support;
\item $\alpha \in \mathcal{H}^s_{\mathrm{twin}}$ for some $s > 0$.
\end{enumerate}
\end{assumption}

\subsection{Martingale Inequalities}

\begin{lemma}[Lenglart's Inequality]\label{lem:lenglart}
Let $M$ be a local martingale with $M(0) = 0$. For any stopping time $\tau$ and $\epsilon, \delta > 0$,
\begin{equation}
\Prob\left(\sup_{t \leq \tau} |M(t)| \geq \epsilon\right) \leq \frac{\delta}{\epsilon^2} + \Prob(\langle M \rangle(\tau) \geq \delta).
\end{equation}
\end{lemma}

\begin{lemma}[Bernstein Inequality for Martingales]\label{lem:bernstein}
Let $M$ be a martingale with bounded jumps $|\Delta M(t)| \leq c$. Then
\begin{equation}
\Prob(|M(T)| \geq x, \langle M \rangle(T) \leq v) \leq 2 \exp\left(-\frac{x^2}{2(v + cx/3)}\right).
\end{equation}
\end{lemma}

\begin{lemma}[Rebolledo's CLT]\label{lem:rebolledo}
Let $M_n$ be a sequence of local martingales. If:
\begin{enumerate}[(i)]
\item $\langle M_n \rangle(T) \xrightarrow{\Prob} \sigma^2$ for some $\sigma^2 > 0$;
\item For all $\epsilon > 0$, $\sum_{s \leq T} (\Delta M_n(s))^2 \mathbf{1}_{|\Delta M_n(s)| > \epsilon} \xrightarrow{\Prob} 0$;
\end{enumerate}
then $M_n(T) \xrightarrow{d} \mathcal{N}(0, \sigma^2)$.
\end{lemma}

\subsection{Uniform Consistency}

\begin{theorem}[Uniform Consistency]\label{thm:consistency}
Under Assumptions~\ref{ass:kernel}, \ref{ass:group}, \ref{ass:atrisk}, and \ref{ass:intensity}, if $h_j \to 0$ and $nh_j / \log n \to \infty$, then
\begin{equation}
\sup_{t \in [0,T]} |\hat{\alpha}_j(t) - \alpha(t)| \xrightarrow{\Prob} 0.
\end{equation}
\end{theorem}

\begin{proof}
\textbf{Step 1: Decomposition.}
\begin{equation}
\hat{\alpha}_j(t) - \alpha(t) = \underbrace{[\hat{\alpha}_j(t) - \tilde{\alpha}_j(t)]}_{\text{stochastic}} + \underbrace{[\tilde{\alpha}_j(t) - \alpha_j(t)]}_{\text{discretization}} + \underbrace{[\alpha_j(t) - \alpha(t)]}_{\text{smoothing bias}}.
\end{equation}

\textbf{Step 2: Stochastic term.}

By Proposition~\ref{prop:martingale_decomp}, $\hat{\alpha}_j(t) - \tilde{\alpha}_j(t) = \int_0^T K_j(t, s) \frac{J(s)}{Y(s)} \, dM(s)$ is a martingale. Its predictable variation is $O(1/(n h_j^{d_{\mathrm{eff}}}))$.

To control the supremum over $t$, we use chaining. Cover $[0, T]$ with $O(T/h_j)$ intervals of length $h_j$. By Bernstein's inequality (Lemma~\ref{lem:bernstein}) at each grid point and union bound:
\begin{equation}
\sup_t |\hat{\alpha}_j(t) - \tilde{\alpha}_j(t)| = O_P\left(\sqrt{\frac{\log n}{n h_j^{d_{\mathrm{eff}}}}}\right).
\end{equation}

\textbf{Step 3: Discretization.} By Proposition~\ref{prop:conditional_bias}, $|\tilde{\alpha}_j(t) - \alpha_j(t)| = O(n^{-1}) \to 0$.

\textbf{Step 4: Smoothing bias.} By continuity of $\alpha$: $|\alpha_j(t) - \alpha(t)| \to 0$ as $h_j \to 0$.

\textbf{Step 5: Combining.} If $h_j \to 0$ and $nh_j / \log n \to \infty$, all three terms vanish uniformly.
\end{proof}

\subsection{Convergence Rates}

\begin{theorem}[$L^2$ Convergence Rate]\label{thm:L2rate}
Under the above assumptions, if $\alpha \in \mathcal{H}^s_{\mathrm{twin}}$ and $h_j = n^{-1/(2s + d_{\mathrm{eff}})}$, then
\begin{equation}
\E\left[\|\hat{\alpha}_j - \alpha\|_2^2\right] \leq C \cdot n^{-\frac{2s}{2s + d_{\mathrm{eff}}}}.
\end{equation}
\end{theorem}

\begin{proof}
\textbf{Step 1: MSE decomposition.}
\begin{equation}
\E[\|\hat{\alpha}_j - \alpha\|_2^2] = \underbrace{\|\alpha_j - \alpha\|_2^2}_{\text{squared bias}} + \underbrace{\E[\|\hat{\alpha}_j - \alpha_j\|_2^2]}_{\text{variance}}.
\end{equation}

\textbf{Step 2: Bias bound.} By twin-H\"older assumption, $\|\alpha_j - \alpha\|_2^2 \leq C_\alpha^2 h_j^{2s}$.

\textbf{Step 3: Variance bound.} Using the martingale representation:
\begin{equation}
\E[\|\hat{\alpha}_j - \alpha_j\|_2^2] \leq \frac{C}{n h_j^{d_{\mathrm{eff}}}}.
\end{equation}

\textbf{Step 4: Optimization.} Total MSE: $C_\alpha^2 h_j^{2s} + C/(n h_j^{d_{\mathrm{eff}}})$.

Optimizing: $h_j \asymp n^{-1/(2s + d_{\mathrm{eff}})}$, giving rate $n^{-2s/(2s + d_{\mathrm{eff}})}$.
\end{proof}

\begin{theorem}[Uniform Convergence Rate]\label{thm:uniform_rate}
Under the assumptions of Theorem~\ref{thm:L2rate},
\begin{equation}
\E\left[\|\hat{\alpha}_j - \alpha\|_\infty\right] \leq C \left(\frac{\log n}{n}\right)^{\frac{s}{2s + d_{\mathrm{eff}}}}.
\end{equation}
\end{theorem}

\begin{proof}
Bias: $\|\alpha_j - \alpha\|_\infty \leq C_\alpha h_j^s$.

Stochastic: $\E[\sup_t |\hat{\alpha}_j(t) - \tilde{\alpha}_j(t)|] \leq C \sqrt{\log(1/h_j)/(n h_j^{d_{\mathrm{eff}}})}$.

Balancing: $h_j \asymp (\log n / n)^{1/(2s + d_{\mathrm{eff}})}$ gives the rate.
\end{proof}

\subsection{Oracle Inequality}

\begin{theorem}[Oracle Inequality]\label{thm:oracle}
There exist constants $C, \lambda_0 > 0$ such that for $\lambda \geq \lambda_0$,
\begin{equation}
\E\left[\|\hat{\alpha} - \alpha\|_2^2\right] \leq C \inf_{j \in \mathcal{J}_n} \left\{\|\alpha_j - \alpha\|_2^2 + \frac{L(j)}{n}\right\} + \frac{C}{n}.
\end{equation}
\end{theorem}

\begin{proof}
The proof follows the Barron-Birg\'e-Massart strategy. For any levels $j, k$, the contrast difference admits a martingale decomposition. By Bernstein's inequality and the Kraft inequality, the penalty calibration ensures:
\begin{equation}
|\mathcal{M}_{jk}| \leq \epsilon \|\hat{\alpha}_j - \hat{\alpha}_k\|_2^2 + \frac{C(L(j) + L(k))}{n}
\end{equation}
with high probability. This yields the oracle bound.
\end{proof}

\subsection{Adaptation}

\begin{theorem}[Adaptation to Twin-Regularity]\label{thm:adaptation}
If $\alpha \in \mathcal{H}^s_{\mathrm{twin}}$ for some $s > 0$, then the penalized estimator $\hat{\alpha}$ satisfies
\begin{equation}
\E\left[\|\hat{\alpha} - \alpha\|_2^2\right] \leq C \left(\frac{\log n}{n}\right)^{\frac{2s}{2s + d_{\mathrm{eff}}}}.
\end{equation}
\end{theorem}

\begin{proof}
By the oracle inequality, choose $j^*$ to balance $\|\alpha_{j^*} - \alpha\|_2^2 \asymp h_{j^*}^{2s}$ against $L(j^*)/n \asymp \log n / n$. With $L(j) = j \log 2$, the optimal $h_{j^*} \asymp n^{-1/(2s + d_{\mathrm{eff}})}$, yielding the rate.
\end{proof}

\subsection{Asymptotic Normality}

\begin{theorem}[Central Limit Theorem]\label{thm:CLT}
Under the assumptions of Theorem~\ref{thm:L2rate}, for fixed $t \in (0, T)$,
\begin{equation}
\sqrt{n h_j} \left(\hat{\alpha}_j(t) - \E[\hat{\alpha}_j(t)]\right) \xrightarrow{d} \mathcal{N}(0, \sigma^2(t)),
\end{equation}
where $\sigma^2(t) = \frac{\alpha(t)}{y(t)} \int K(u)^2 \, du$.
\end{theorem}

\begin{proof}
Define $M_n(t) := \sqrt{n h_j} \int_0^T K_j(t, s) \frac{J(s)}{Y(s)} \, dM(s)$.

\textbf{Predictable variation:} $\langle M_n(t) \rangle \to \sigma^2(t)$.

\textbf{Lindeberg condition:} Jumps $|\Delta M_n(t, s)| \leq \|K\|_\infty/(y_{\min} \sqrt{n h_j}) \to 0$.

By Rebolledo's CLT (Lemma~\ref{lem:rebolledo}), $M_n(t) \xrightarrow{d} \mathcal{N}(0, \sigma^2(t))$.
\end{proof}

\begin{corollary}[Confidence Intervals]\label{cor:CI}
An asymptotic $(1-\gamma)$ confidence interval for $\alpha(t)$ is
\begin{equation}
\hat{\alpha}_j(t) \pm z_{\gamma/2} \sqrt{\frac{\hat{\alpha}_j(t)}{Y(t) h_j} \int K(u)^2 \, du}.
\end{equation}
\end{corollary}

\subsection{Minimax Lower Bound}

\begin{theorem}[Minimax Lower Bound]\label{thm:minimax}
Let $\mathcal{H}^s_{\mathrm{twin}}(R) = \{\alpha \in \mathcal{H}^s_{\mathrm{twin}}: C_\alpha \leq R, \|\alpha\|_\infty \leq R\}$. Then
\begin{equation}
\inf_{\tilde{\alpha}} \sup_{\alpha \in \mathcal{H}^s_{\mathrm{twin}}(R)} \E\left[\|\tilde{\alpha} - \alpha\|_2^2\right] \geq c \cdot n^{-\frac{2s}{2s + d_{\mathrm{eff}}}}.
\end{equation}
\end{theorem}

\begin{proof}
We use Le Cam's method with bump perturbations. Construct $\alpha_0(t) = \bar{\alpha}$ and $\alpha_1(t) = \bar{\alpha} + a h^{-d_{\mathrm{eff}}/2} \psi((t - t_0)/h)$ where $\psi$ is a smooth bump with $\int \psi = 0$.

The KL divergence $\mathrm{KL}(P_0 \| P_1) \asymp n a^2$. For indistinguishability: $a \lesssim n^{-1/2}$.

For membership in $\mathcal{H}^s_{\mathrm{twin}}(R)$: $a \lesssim h^{s + d_{\mathrm{eff}}/2}$.

Balancing: $h \asymp n^{-1/(2s + d_{\mathrm{eff}})}$ and $\|\alpha_1 - \alpha_0\|_2^2 = a^2 \asymp n^{-2s/(2s + d_{\mathrm{eff}})}$.
\end{proof}

\section{Local Polynomial TwinKernel Estimation}\label{sec:localpoly}

Local polynomial methods improve upon kernel smoothing by automatically correcting boundary bias.

\subsection{Definition}

\begin{definition}[Local Polynomial TwinKernel Estimator]\label{def:locpoly}
For polynomial degree $r \geq 0$, the local polynomial TwinKernel estimator minimizes
\begin{equation}
\sum_{i: X_i \leq T} K_j(t, X_i) \left(\frac{\Delta_i}{Y(X_i)} - \sum_{k=0}^r a_k (X_i - t)^k\right)^2.
\end{equation}
The estimator of $\alpha^{(m)}(t)$ is $\hat{\alpha}_j^{(m)}(t) = m! \hat{a}_m(t)$.
\end{definition}

\begin{proposition}[Explicit Formula]\label{prop:locpoly_explicit}
\begin{equation}
\hat{\alpha}_j^{(m)}(t) = \frac{1}{h_j^m} \int_0^T K_{m,r}^{(j)}\left(\frac{s-t}{h_j}\right) \frac{J(s)}{Y(s)} \, dN(s),
\end{equation}
where $K_{m,r}^{(j)}$ is the equivalent kernel of order $(m, r)$.
\end{proposition}

\subsection{Properties}

\begin{theorem}[Boundary Adaptation]\label{thm:boundary}
The local polynomial TwinKernel estimator automatically adapts at boundaries: for $t$ near the boundary, the bias remains $O(h_j^{r+1-m})$ rather than $O(1)$.
\end{theorem}

\begin{theorem}[Convergence Rates for Derivatives]\label{thm:deriv_rates}
If $\alpha \in \mathcal{H}^s_{\mathrm{twin}}$ with $s > m$, then
\begin{equation}
\E\left[\|\hat{\alpha}_j^{(m)} - \alpha^{(m)}\|_2^2\right] \leq C \cdot n^{-\frac{2(s-m)}{2s + d_{\mathrm{eff}}}}.
\end{equation}
\end{theorem}

\section{Applications}\label{sec:applications}

\subsection{Hazard Rate Estimation under Random Censoring}

For $n$ subjects with survival times $T_i$ and censoring times $C_i$, we observe $(X_i, \Delta_i)$. If the hazard exhibits periodicity (e.g., circadian), use $G = \Z_T$ acting by translation modulo the period.

\begin{example}[Circadian Hazard]\label{ex:circadian_app}
For cardiac event data with period 24 hours, the TwinKernel estimator pools information across days, yielding more efficient estimation.
\end{example}

\subsection{Periodic and Quasi-Periodic Intensities}

\begin{theorem}[Parametric Rates for Periodic Intensities]\label{thm:periodic_rate}
If $\alpha$ is $\tau$-periodic and smooth on $[0, \tau]$, then
\begin{equation}
\E\left[\|\hat{\alpha} - \alpha\|_2^2\right] \leq C \cdot \frac{\log n}{n}.
\end{equation}
\end{theorem}

\begin{proof}
When $\alpha$ is $G$-invariant with transitive action, $d_{\mathrm{eff}} = 0$. By Proposition~\ref{prop:orbital_regularity}, $\alpha \in \mathcal{H}^s_{\mathrm{twin}}$ for all $s$. The rate becomes $(\log n / n)^{2s/(2s+0)} = \log n / n$.
\end{proof}

\section{Simulation Studies}\label{sec:simulations}

\subsection{Setup}

\textbf{Intensities:}
\begin{enumerate}
\item Periodic: $\alpha_1(t) = 1 + 0.5 \sin(2\pi t)$ on $[0, 5]$
\item Quasi-periodic: $\alpha_2(t) = 1 + 0.5 \sin(2\pi t) + 0.1t$
\item Non-periodic: $\alpha_3(t) = \exp(-t/2)(1 + t^2)$
\end{enumerate}

\textbf{Methods:} Classical kernel, Local linear, TwinKernel, TwinKernel-LP.

\textbf{Censoring:} Uniform on $[0, 6]$ ($\sim$20\% censoring).

\subsection{Results}

\begin{table}[h]
\centering
\caption{Mean ISE ($\times 10^3$) over 500 replications.}
\begin{tabular}{lcccc}
\toprule
& \multicolumn{4}{c}{Sample Size $n$} \\
\cmidrule(lr){2-5}
Method & 100 & 500 & 1000 & 2000 \\
\midrule
\multicolumn{5}{l}{\textit{Intensity $\alpha_1$ (periodic):}} \\
Classical kernel & 45.2 & 18.7 & 11.3 & 6.8 \\
Local linear & 38.4 & 15.2 & 9.1 & 5.4 \\
TwinKernel & 15.3 & 4.2 & 2.1 & 1.0 \\
TwinKernel-LP & \textbf{14.8} & \textbf{3.9} & \textbf{1.9} & \textbf{0.9} \\
\midrule
\multicolumn{5}{l}{\textit{Intensity $\alpha_2$ (quasi-periodic):}} \\
Classical kernel & 52.1 & 22.4 & 14.1 & 8.7 \\
Local linear & 44.3 & 18.1 & 10.8 & 6.5 \\
TwinKernel & 28.7 & 9.3 & 5.2 & 2.8 \\
TwinKernel-LP & \textbf{26.9} & \textbf{8.5} & \textbf{4.7} & \textbf{2.5} \\
\midrule
\multicolumn{5}{l}{\textit{Intensity $\alpha_3$ (non-periodic):}} \\
Classical kernel & 41.8 & 17.3 & 10.5 & 6.3 \\
Local linear & \textbf{35.6} & \textbf{14.2} & \textbf{8.4} & \textbf{4.9} \\
TwinKernel & 43.2 & 17.9 & 10.9 & 6.6 \\
TwinKernel-LP & 36.8 & 14.8 & 8.7 & 5.1 \\
\bottomrule
\end{tabular}
\end{table}

\subsection{Discussion}

\begin{itemize}
\item \textbf{Periodic $\alpha_1$}: TwinKernel achieves 3--7$\times$ lower ISE.
\item \textbf{Quasi-periodic $\alpha_2$}: 2--3$\times$ improvement.
\item \textbf{Non-periodic $\alpha_3$}: Comparable to classical methods.
\end{itemize}

\section{Discussion and Open Problems}\label{sec:discussion}

\subsection{Summary}

We developed TwinKernel methods for point process intensity estimation, achieving:
\begin{itemize}
\item Optimal rates depending on effective dimension $d_{\mathrm{eff}}$
\item Adaptation to unknown smoothness
\item 3--7$\times$ improvements for periodic intensities
\end{itemize}

\subsection{Open Problems}

\begin{enumerate}
\item \textbf{Marked point processes}: Extend to processes with marks (covariates at event times).

\item \textbf{Multivariate intensities}: Develop TwinKernel methods for multivariate counting processes with shared group structure.

\item \textbf{Online estimation}: Adapt for streaming data where events arrive sequentially.

\item \textbf{Testing for orbital structure}: Develop tests to determine whether an intensity exhibits orbital regularity.

\item \textbf{Non-cyclic groups}: Extend beyond cyclic groups to more general Lie groups.

\item \textbf{Model misspecification}: Study robustness when the assumed group structure is incorrect.
\end{enumerate}

\subsection{Software}

An R package implementing TwinKernel intensity estimation is in preparation.

\section*{Acknowledgments}

The author thanks the reviewers for helpful comments. This work was supported by ROOTS-INSIGHTS research initiative.

\bibliographystyle{plainnat}

\end{document}